\documentclass[11pt,reqno]{amsart}

\usepackage{amsmath,amssymb,mathtools}
\usepackage[margin=1.15in]{geometry}
\usepackage{microtype,enumerate,comment}
\usepackage[hidelinks]{hyperref}
\usepackage[new]{old-arrows}

\newtheorem{theorem}{Theorem}[section]
\newtheorem{alphatheorem}{Theorem}

\newtheorem{proposition}[theorem]{Proposition}
\newtheorem{corollary}[theorem]{Corollary}
\newtheorem{lemma}[theorem]{Lemma}
\theoremstyle{definition}
\newtheorem{definition}[theorem]{Definition}
\theoremstyle{remark}

\newcommand{\Z}{\mathbb Z}
\newcommand{\Q}{\mathbb Q}
\newcommand{\Gm}{{\mathbb G}_{\mathrm m}}
\newcommand{\OK}{\mathcal O_K}
\newcommand{\Tr}{\operatorname{Tr}}
\newcommand{\Nm}{\operatorname{N}}
\newcommand{\ord}{\operatorname{ord}}
\newcommand{\Res}{\operatorname{Res}}
\newcommand{\Gal}{\operatorname{Gal}}
\newcommand{\eps}{\varepsilon}
\newcommand{\leg}[2]{({\scriptstyle{\frac{#1}{#2}}})}

\def\pp{\mathfrak{p}}
\def\rad{\operatorname{rad}}
\def\ayz{\mathbf A^1_{\Z}}
\def\qD{\sqrt{D}}

\begin{document}

\title{SIC dimension towers via cyclotomic polynomials}

\author{Gary McConnell}

\thanks{Controlled Quantum Dynamics Theory Group, Imperial College,
London SW7 2AZ, UK.
Email: \href{mailto:gimcc@ic.ac.uk}{\nolinkurl{gimcc@ic.ac.uk}}.}

\begin{abstract}
We prove a structure theorem for the dimension towers~$\{d_k(D)\}_{k\geq0}$ 
which arise in the number-theoretic formulation of 
Zauner's SIC-POVM conjecture over a real quadratic
field~$K=\Q(\qD)$.  If~$\eps$ denotes the first totally positive power 
of a fundamental unit of~$K$ and $t_k = \eps^k + \eps^{-k}$ 
the trace of its $k$-th power, then the SIC 
dimension tower~$\{d_k=1+t_k\}_{k\geq0}$ 
is the level~$m=3$ row of an infinite two-dimensional
cyclotomic array~$\{\Psi_m(t_k)\}_{m\geq1,k\geq0}$ 
attached to~$K$, while the auxiliary factors~$(d_k+1)$ and
$(d_k-3)$ are its ramified levels $m=2$ and $m=1$.  
Here~$\Psi_m$ denotes the minimal polynomial of~
$\zeta_m+\zeta_m^{-1} = 2\cos{2\pi/m}$. 
This construction arose initially from an attempt to formulate
relations among SIC dimensions in $q$-algebraic terms. 

The central object is a single closed composite norm relation for the
two-parameter family $c_{m,k}=1-\zeta_m\eps^k$ 
over the cyclotomic field tower 
$\{K(\mu_m)\}_{m\geq1}$.  This framework sheds new light on 
the mod-$p$ analogue of Leopoldt's conjecture, 
by placing the central 3-symmetry of
Zauner's conjecture within a broader arithmetic context. 
Away from the primes dividing~$2mD$, the valuations
$v_p(\Psi_m(t_k))$ at every fixed level~$m$ are described exactly in
terms of a single local unit valuation, which is then related, through
the~$p$-adic class number formula, to the corresponding $p$-adic
$L$-value.
\end{abstract}

\makeatletter
\let\saved@makefntext\@makefntext
\def\@makefntext{\noindent\@makefnmark}
\makeatother

\maketitle

\makeatletter
\let\@makefntext\saved@makefntext
\makeatother

\section{Introduction}
Let~$D>1$ be a square--free integer, and 
write~$K=\Q(\qD)$ for the corresponding real quadratic field. 
$\Delta = D$ or~$4D$ will denote its discriminant.
Let $u$ be a fundamental unit of~$K$: to fix ideas 
once and for all we assume that~$u>1$ 
at the real embedding for which~$\qD>0$. 
Let $p>2$ be a prime number. 
Meaningful calculations---let alone the theory---for 
the problem variously referred to~\cite{McC24} as 
\emph{Wieferich} or \emph{Wall--Sun--Sun} primes 
for real quadratic fields, or 
the mod-$p$ analogue of Leopoldt's conjecture~\cite{BGKK},
are widely regarded as intractable with current 
methods~\cite{KatzWieferich,silverman}.

However the number-theoretic formulation~\cite{AFMY,Kopp,KLagSIC} 
of Zauner's SIC-POVM existence conjecture~\cite{Zauner}---which associates
to a dimension $d\geq4$ the class 
\( D \equiv (d+1)(d-3) \bmod {\Q^\times}^2 \)---
shows these in a new light~\cite{McC24}. 
In its simplest form, it is the observation that any dimension 
with few divisors and a ``large'' prime power factor $p^2$ or higher will 
force~$K$ to be a non-$p$-rational field \emph{via local $p$-divisibility of the unit 
group}~\cite{grasNpR}, as opposed to the route via the Hilbert class group. 
Appendix~\ref{app:wss-tables} gives exhaustive tables for
$1<D<50$ and $p<1.7415\times10^{13}$, while
Appendix~\ref{app:null-tests} summarizes the associated computational
tests (see also~\cite{primegrid,WSSwiki}).

The initial clue came from the divisibility of the two auxiliary
factors: with $q=\eps^k$,
$d_k-3=q^{-1}(q-1)^2$ and $d_k+1=q^{-1}(q+1)^2$.
Studying $1-\zeta_mq$ with $q$ indeterminate, before specializing
to $q=\eps^k$, identifies these as the levels $m=1,2$ of a general
cyclotomic construction and makes both the higher factorizations and
the substitutions $q\mapsto q^N$ transparent. 
This Laurent-polynomial~$q$-algebra viewpoint both 
motivated and substantially simplified the argument.

In this paper we place the familiar SIC dimension tower~\cite{AFMY} inside an
infinite integer array attached to $K$ via a cyclotomic construction.  
At $k=0$ the elements~$c_{m,0}=1-\zeta_m$ 
are the classical cyclotomic values of Kubert's 
\emph{universal ordinary distribution}~\cite{Kubert79}.
For $k\geq1$, the same cyclotomic distribution relation persists
after twisting by $\varepsilon^k$, accompanied in the $k$-variable by
the Chebyshev shift $t_k\mapsto t_{Nk}=2T_N(t_k/2)$. In the concluding
remarks in \S\ref{sec:conclusions}
we show that, after adjoining compatible Kummer division points, the
prime-power relation becomes a genuinely two-directional norm relation
in the Kummer--cyclotomic setting, in what is usually referred to 
as a \emph{false Tate extension}~\cite[\S3.3]{HV03}.

\subsection{Cyclotomic norm relations}\label{subsec:main-theorem}
Henceforth we write~$\tau$ for the non-trivial automorphism of~$K/\Q$. 
Let~$\eps$ denote the first totally positive power of $u$: 
\[
 \eps:=
 \begin{cases}
  u,&\Nm_{K/\Q}(u)=1,\\
  u^2,&\Nm_{K/\Q}(u)=-1;
 \end{cases}
\]
so in particular~$\Nm_{K/\Q}(\eps)=1$ and~$\tau(\eps)=\eps^{-1}$.  
Thus the powers of $\eps$ have integer traces
\begin{equation}\label{eq:trace-sequence}
 t_k:=\Tr_{K/\Q}(\eps^k)=\eps^k+\eps^{-k},
 \qquad k\geq0.
\end{equation}
The point $t_0=2$ is the natural initial
term of the sequence, for every~$D$. 
In the SIC context, we recall that for~$k\geq0$
we define the $k$-th dimension in the tower~\cite{AFMY} 
above $K$ to be~$d_k := d_k(D) := t_k+1 = \eps^k + \eps^{-k} + 1$; 
consequently~$d_0(D)=3$ is the base point for every~$D$.

Fix once and for all an algebraic closure $\overline{\Q}$ of $\Q$ 
and a power-compatible system of primitive roots of
unity~$\{\zeta_m\in\overline{\Q}:m\geq1\}$, meaning that
$\zeta_{mn}^{\,n}=\zeta_m$ for all $m,n\geq1$.  
Loosely, to fix ideas, we may regard these as simply 
$\zeta_m = e^\frac{2 \pi i}{m}$, and therefore 
the ubiquitous trace quantity $\zeta_m+\zeta_m^{-1} = 2\cos \frac{2\pi }{m}$.
Write $\mu_m$ for the group of $m$-th roots of unity.
For $m\geq1$ and $k\geq0$, define the following natural two-index
family of elements over~$K$:
\begin{equation}\label{eq:cmk}
 c_{m,k}:=1-\zeta_m\eps^k\in K(\mu_m).
\end{equation}
Note that $c_{m,k} = 0$ precisely when $(m,k) = (1,0)$.
For any~$a\in(\Z/m\Z)^\times$, let $\sigma_a$ denote the Galois
automorphism of $\Q(\mu_m)/\Q$ which sends $\zeta_m$ to
$\zeta_m^a$. 
Since $K$ is quadratic, for any $N\geq1$ the intersection
$K\cap\Q(\mu_N)$ is either $\Q$ or $K$; the cyclotomic norm statements
below concern the generic case $K\cap\Q(\mu_N)=\Q$.  The minor
modifications required when $K\subseteq\Q(\mu_m)$ are recorded in
Appendix~\ref{app:intersection}. 
Whenever $K\cap\Q(\mu_m)=\Q$, we use the same notation~$\sigma_a$
for the induced automorphism of $K(\mu_m)/K$ which fixes~$K$. 

Let~$\Phi_m(X)$ denote the $m$-th cyclotomic polynomial, and let
$\Psi_m(X)\in\Z[X]$ be the minimal polynomial over~$\Q$ of
$\zeta_m+\zeta_m^{-1}$.  Write $\varphi$ and $\mu$ for Euler's
totient function and the M\"obius function.  For $m\geq3$, complex
conjugation pairs the primitive $m$-th roots without fixed points; in
particular~\cite{Washington}, $\Psi_m$ has degree $\varphi(m)/2$ and
its roots are $\zeta_m^a+\zeta_m^{-a}$, with $a$ taken modulo the
pairing $a\sim-a$ in $(\Z/m\Z)^\times$.  Pairing the corresponding
roots of $\Phi_m$ gives
\begin{equation}\label{eq:pal}
 \Phi_m(X)=X^{\varphi(m)/2}\Psi_m(X+X^{-1})
 \qquad(m\geq3).
\end{equation}
Let $P_t(Y):=Y^2-tY+1$.  If $z+z^{-1}=t$, then $z,z^{-1}$ are the
roots of $P_t$, and therefore
\begin{equation}\label{eq:resultant}
 \Res_Y\bigl(P_t(Y),\Phi_m(Y)\bigr)
 =\Phi_m(z)\Phi_m(z^{-1})=\Psi_m(t)^2
 \qquad(m\geq3).
\end{equation}
Thus $\Psi_m$ may equivalently be defined to be 
the unique monic square root of this
resultant. It is the exact order-$m$ cyclotomic factor after passage
from the multiplicative coordinate $z$ to the trace coordinate
$t=z+z^{-1}$.  At the bottom levels the same minimal-polynomial
convention gives $\Psi_1(t)=t-2$ and $\Psi_2(t)=t+2$; these will be
the two branch levels of the quotient in~\S\ref{sec:geometry}.

\begin{definition}\label{def:level}
With $D$ understood, the expression~\emph{level}~$m$ will mean the row
$\{\Psi_m(t_k)\}_{k\geq0}$ of the array.  Since $\Psi_3(t)=t+1$,
level~$3$ is the SIC tower~$\{d_k\}_{k\geq0}$. 
Every entry of the array is a rational integer, since $t_k\in\Z$ and
$\Psi_m\in\Z[X]$. 
\end{definition}

For Theorem~\ref{thm:main} and the identities derived from it in
\S\ref{sec:norms}, given integers~$m,n\geq1$, let $n=bn'$, where
$(n',m)=1$ and every prime divisor of $b$ divides $m$.
Write $\rad{N}$ for the radical of an integer~$N$.
Whenever we have an integer $e$ such that $(e,m)=1$, we write $e^{-1}$ for the inverse of $e$ modulo $m$.
When $m=1$, every coefficient automorphism so indexed is understood to be the identity.

\begin{theorem}\label{thm:main}
Let $m,n\geq1$ and $k\geq0$, with $(m,k)\neq(1,0)$.  Write
$n=bn'$, where $(n',m)=1$ and every prime divisor of $b$ divides $m$,
and interpret $e^{-1}$ as above.  Assume that
$K\cap\Q(\mu_{mn})=\Q$.  Then
\begin{equation}\label{eq:composite-relation}
 \Nm_{K(\mu_{mn})/K(\mu_m)}\bigl(c_{mn,k}\bigr)
 =\prod_{e\mid\rad{n'}}
  \sigma_{e^{-1}}\bigl(c_{m,(n/e)k}\bigr)^{\mu(e)}.
\end{equation}
\end{theorem}

We shall prove this in \S\ref{sec:norms}, by composing
elementary norm identities of cyclotomic polynomials which move one
prime at a time through the tower --- incrementing an $\ell$-power
level, or adjoining a new prime~$\ell$ to an existing level.  The
single-prime cases of the theorem itself are as follows. 

\begin{corollary}[Prime-step relations]\label{cor:prime-step}
Let $\ell$ be a prime and let $k\geq0$. 
\begin{enumerate}[(i)]
\item\label{cor:prime-step-deepening}
If $\ell$ is odd, then for every $s\geq1$,
\[
  \Nm_{K(\mu_{\ell^{s+1}})/K(\mu_{\ell^s})}
  \bigl(c_{\ell^{s+1},k}\bigr)
  =
  c_{\ell^s,\ell k}.
\]

\item\label{cor:prime-step-adjoining}
Let $m\geq1$ satisfy $(m,\ell)=1$ and assume that~$K\cap\Q(\mu_{m\ell})=\Q$. 
\begin{enumerate}[(a)]
\item\label{cor:prime-step-adjoining-general}
If $m\geq2$, then
\[
  \Nm_{K(\mu_{m\ell})/K(\mu_m)}
  \bigl(c_{m\ell,k}\bigr)
  =
  \frac{c_{m,\ell k}}
       {\sigma_{\ell^{-1}}(c_{m,k})}.
\]
\item\label{cor:prime-step-adjoining-bottom}
If $m=1$ then $\Nm_{K(\mu_\ell)/K}\bigl(c_{\ell,k}\bigr) = \Phi_\ell(\eps^k)$. 
If~$k\geq1$, this may also be written as
\[
  \Nm_{K(\mu_\ell)/K}\bigl(c_{\ell,k}\bigr)
  =
  \frac{c_{1,\ell k}}{c_{1,k}}.
\]
\end{enumerate}
\end{enumerate}
\end{corollary}

\subsection{The arithmetic of the SIC dimension towers}\label{subsec:structure-theorem}
The principal conclusions of the paper are collected in a form which
emphasizes the full cyclotomic array rather than its level-$3$ specialization. 
For a prime $p\nmid2D$, let $\bar\eps_p$ denote the 
reduction of~$\eps$ in the norm-one 
subgroup~$\{x\in(\OK/p\OK)^\times:x x^\tau=1\}$ of 
order~$p-\leg{D}{p}$, where~$\leg{D}{p}$ 
is the Legendre symbol, and put
$o(p):=\ord(\bar\eps_p)$.
For any prime $\mathfrak p\mid p\mathcal O_K$, let
$v_{\mathfrak p}$ be normalized by $v_{\mathfrak p}(p)=1$, and put
\[
 a_p:=v_{\mathfrak p}\bigl(\eps^{o(p)}-1\bigr).
\]
This is independent of $\mathfrak p$: in the inert case there is only
one prime above $p$, while in the split case $\tau$ interchanges the
two primes and
\[
 \tau\bigl(\eps^{o(p)}-1\bigr)
 =\eps^{-o(p)}-1
 =-\eps^{-o(p)}\bigl(\eps^{o(p)}-1\bigr).
\]
For fixed~$m\geq1$ and
$p\nmid2mD$, define the \emph{rank of apparition of~$p$} 
to be~$\rho_m(p):=\min\{k\geq1:p\mid\Psi_m(t_k)\}$,
where we adopt the convention that~$\rho_m(p)=\infty$ if no such index exists.

\begin{alphatheorem}\label{thm:A}
\begin{enumerate}[(i)]
\item\label{thmA:parametrization}
The assignment $(D,k)\longmapsto t_k(D)$ is a bijection from the pairs
$D>1$ square--free, $k\geq1$, onto $\Z_{\geq3}$.  Consequently, for every fixed
$m\geq3$, the map $(D,k)\longmapsto\Psi_m(t_k(D))$ is injective.  Its values
exhaust every integer $N\geq\Psi_m(3)$ if and only if $m\in\{3,4,6\}$; at
every other level they form an infinite set of density zero.

\item\label{thmA:chebyshev}
For every fixed square--free $D>1$ and all $n,k\geq1$, with $\mathbf{1}_{2\mid n}$
denoting the indicator of $2\mid n$,
\[
 t_{nk}-2
 =(t_k-2)(t_k+2)^{\mathbf{1}_{2\mid n}}
  \prod_{\substack{m\mid n\\ m\geq3}}\Psi_m(t_k)^2.
\]

\item\label{thmA:support}
Fix square--free $D>1$.  Let $m\geq3$ and $p\nmid2mD$.  Then
$\rho_m(p)<\infty$ if and only if $m\mid o(p)$, in which case
$\rho_m(p)=o(p)/m$.  Moreover, $p\mid\Psi_m(t_k)$ if and only if
$k=\rho_m(p)l$ for some $l\geq1$ coprime to $m$. 
At every such occurrence,~$v_p(\Psi_m(t_k))=a_p+v_p(l)$.

\item\label{thmA:orders}
For every fixed square--free $D>1$ and every $n>12$ there is a prime
$p\nmid2D$ with $o(p)=n$.
\end{enumerate}
\end{alphatheorem}

In part~\textup{(\ref{thmA:chebyshev})}, the factors $t_k-2$ and
$t_k+2$ are the branch levels $m=1,2$.  At level $m=3$,
parts~\textup{(\ref{thmA:parametrization})} and
\textup{(\ref{thmA:support})} recover the familiar SIC parametrization
and rank-of-apparition statements~\cite{AFMY,BGM}. 
Theorem~\ref{thm:main} supplies the cyclotomic distribution relations
used throughout.

\section[The quotient of Gm by inversion and its quadratic invariant]{The quotient of \texorpdfstring{\(\Gm\)}{Gm} by inversion and its quadratic invariant}\label{sec:geometry}
\label{subsec:trace-quotient}
Here and below, \(\Gm\) denotes the multiplicative group; the roman
`m' distinguishes its subscript from the cyclotomic level~\(m\).  The
involution \(\iota(z)=z^{-1}\) has invariant coordinate
\(t=z+z^{-1}\), and
\(\Z[z,z^{-1}]^{\langle\iota\rangle}=\Z[t]\).  Thus
\begin{equation}\label{eq:GmAfline}
 \pi:\Gm\longrightarrow\ayz,
 \qquad z\longmapsto t=z+z^{-1}
\end{equation}
is the quotient by inversion. 
Equivalently, \(P_t(z)=0\), and the corresponding homomorphism of
coordinate rings is
\begin{equation}\label{eq:coords}
 \Z[t]\longhookrightarrow\Z[z,z^{-1}]
   \cong\Z[t][z]/(z^2-tz+1).
\end{equation}
The ring on the right is a free \(\Z[t]\)-module of rank~\(2\) with
basis \(1,z\); hence the morphism of schemes \(\pi\) is finite locally
free of rank~\(2\).  For \(b\in\Z[z,z^{-1}]\), define
\(\Nm_\pi(b)\) to be the determinant over \(\Z[t]\) of multiplication
by \(b\), which is independent of the chosen
basis.  Since \(z^{-1}=t-z\), a direct calculation in the basis
\(1,z\) gives
\[
 \Nm_\pi(f(z))=f(z)f(z^{-1})
 =\Res_X\bigl(P_t(X),f(X)\bigr)
 \qquad(f\in\Z[X]).
\]
For \(m\geq3\), after base change to \(\Z[1/m]\), the equation
\(\Phi_m(z)=0\) defines a finite étale closed subscheme of
\(\Gm\), whose geometric points are precisely the roots of unity of
exact order~\(m\).  Inversion acts freely on this subscheme, and its
quotient under~\(\pi\) is the finite étale subscheme defined by
\(\Psi_m(t)=0\).  Indeed, with the usual variety notation, \eqref{eq:pal} gives
\[
 V(\Phi_m)=\pi^{-1}\bigl(V(\Psi_m)\bigr),
\]
since \(z\) is a unit.  Correspondingly, we obtain the exact 
resultant identity~\eqref{eq:resultant}:
\[
 \Nm_\pi\bigl(\Phi_m(z)\bigr)
 =\Phi_m(z)\Phi_m(z^{-1})
 =\Psi_m(t)^2. 
\]

We next record where the quadratic morphism~$\pi$ fails to be
\'etale.  The discriminant of~$P_t$ is $t^2-4$.  Since the extension
\eqref{eq:coords} is monogenic, its different is generated by
\begin{equation}\label{eq:different-bottom-levels}
 P_t'(z)=z-z^{-1}=z^{-1}(z-1)(z+1)
                 =z^{-1}\Phi_1(z)\Phi_2(z).
\end{equation}
Thus the non-\'etale locus of~$\pi$ is the fixed-point subscheme
$z^2=1$.  After base change to $\Z[1/2]$, this subscheme is finite
\'etale over the base and splits as the disjoint union of the two
sections $z=\pm1$.  The morphism~$\pi$ is ramified along these
sections, whose images are the branch sections $t=\pm2$, with
\begin{equation}\label{eq:bottom-branch}
 (z-1)^2=z(t-2),\qquad (z+1)^2=z(t+2),
\end{equation}
or equivalently
\begin{equation}\label{eq:bottom-pullback}
 \Psi_1(t)=z^{-1}\Phi_1(z)^2,\qquad
 \Psi_2(t)=z^{-1}\Phi_2(z)^2.
\end{equation}
Thus $m=1,2$ are exactly the two branch levels.  Specializing
\eqref{eq:bottom-branch} at $z=\eps^k$ gives
\begin{equation}\label{eq:bottom}
 (\eps^k-1)^2=\eps^k(t_k-2),\qquad
 (\eps^k+1)^2=\eps^k(t_k+2).
\end{equation}
Let $T_n$ and $U_n$ denote the Chebyshev polynomials of the first and
second kinds.  The sequences $2T_n(t/2)$ and $U_n(t/2)$ have initial
pairs $2,t$ and $1,t$, respectively, and both satisfy the 
recurrence~$F_{n+1}(t) = t F_n(t) - F_{n-1}(t)$.
Hence~$2T_n(t/2)$ and~$U_n(t/2)$ lie in~$\Z[t]$ for all~$n\geq0$, 
so their values at every integral trace coordinate $t_k$ are integers. 
Moreover, $T_{ab}=T_a\circ T_b$ for any positive integers~$a,b$.  Since
$z\mapsto z^n$ commutes with inversion, the following diagram is
commutative: 
\begin{equation}\label{commdiag}
\begin{array}{ccc}
\Gm & \xrightarrow{\ z\mapsto z^n\ } & \Gm\\
\pi\downarrow\phantom{\pi} && \phantom{\pi}\downarrow\pi\\
\ayz & \xrightarrow{\ t\mapsto 2T_n(t/2)\ } & \ayz
\end{array} 
\end{equation}
For fixed~$m$, diagram~\eqref{commdiag} records the transformation
$k\mapsto nk$ in the family~\eqref{eq:cmk}: evaluating
$1-\zeta_m z^n$ at $z=\eps^k$ gives $c_{m,nk}$, while the lower
horizontal map sends $t_k$ to $t_{nk}=2T_n(t_k/2)$.  Thus the
Chebyshev shift accompanying the cyclotomic norm relations is simply
what the power substitution $z\mapsto z^n$ becomes after passage to
the quotient by inversion.  For $n=\ell$, this is exactly the
$k\mapsto\ell k$ shift in Corollary~\ref{cor:prime-step}.

\subsection{The quadratic field carried by the cyclotomic array}\label{subsec:cross-level}
One has the standard identity~$z^n-z^{-n}=(z-z^{-1})U_{n-1}(t/2)$, and hence
\begin{equation}\label{eq:chew}
 \bigl(2T_n(t/2)\bigr)^2-4=(t^2-4)U_{n-1}(t/2)^2.
\end{equation}
For $k\geq1$, equation~\eqref{eq:chew}, applied at $t=t_1$, 
gives~$t_k^2-4=(t_1^2-4)U_{k-1}(t_1/2)^2$. 
Thus all $t_k^2-4$ have the same square--free part. 
At level~$3$ this gives the familiar 
formula~$D \equiv (d_k-3)(d_k+1) \bmod {\Q^\times}^2$. 

Conversely, for $t\geq3$ let $D$ be the square--free part of~$t^2-4$.  Then~$
\frac{t+\sqrt{t^2-4}}{2}$ is a totally positive norm-one unit of $\Q(\sqrt D)$ with trace~$t$.
Since the totally positive norm-one units of $\Q(\sqrt D)$ form the
group $\langle\eps\rangle$, this unit equals $\eps^k$ for a unique
$k\geq1$.  Hence
$(D,k)\mapsto t_k(D)$ is a bijection onto $\Z_{\geq3}$.

For $m\geq3$, write the roots of the monic polynomial $\Psi_m$ as
$\alpha_1,\ldots,\alpha_{\varphi(m)/2}$.  Since every $\alpha_j$ lies
in $(-2,2)$, for $x\geq2$,
\[
 \Psi_m'(x)=\sum_i\prod_{j\neq i}(x-\alpha_j)>0.
\]
Thus $\Psi_m$ is strictly increasing on $[2,\infty)$, and the map
$(D,k)\mapsto\Psi_m(t_k(D))$ is injective.  Since
$\deg\Psi_m=\varphi(m)/2$ and $\Psi_m$ is monic,
\[
 \#\bigl(\Psi_m(\Z_{\geq3})\cap[1,X]\bigr)
 =\#\{t\geq3:\Psi_m(t)\leq X\}
 \sim X^{2/\varphi(m)}.
\]
For $m\in\{3,4,6\}$ the degree is~$1$, so the values exhaust all
integers from $\Psi_m(3)$ onwards.  At every other level the degree is
at least~$2$, and the displayed counting function is $o(X)$; hence
$\Psi_m(\Z_{\geq3})$ has natural density zero.  This proves
Theorem~\ref{thm:A}\textup{(\ref{thmA:parametrization})}.

\subsection[The Chebyshev factorization in the t-coordinate]{The Chebyshev factorization in the \texorpdfstring{$t$}{t}-coordinate}\label{app:chebyshev-bands}
Since $t_{nk}=2T_n(t_k/2)$, the exact-order factorization of $X^n-1$
gives the following identity.

\begin{proposition}\label{prop:bands}
For every $n\geq1$,
\begin{equation}\label{eq:bands}
 2T_n(t/2)-2
 =(t-2)(t+2)^{\mathbf{1}_{2\mid n}}
 \prod_{\substack{m\mid n\\m\geq3}}\Psi_m(t)^2
 \qquad\text{in }\Z[t].
\end{equation}
Consequently, for every $k\geq1$,
\begin{equation}\label{eq:dimensionbands}
 d_{nk}-3
 =(d_k-3)(d_k+1)^{\mathbf{1}_{2\mid n}}
 \prod_{\substack{m\mid n\\m\geq3}}\Psi_m(t_k)^2.
\end{equation}
\end{proposition}

\begin{proof}
Put $t=X+X^{-1}$.  Then
$2T_n(t/2)-2=X^{-n}(X^n-1)^2$.  Using
$X^n-1=\prod_{m\mid n}\Phi_m(X)$, apply \eqref{eq:pal} for $m\geq3$
and \eqref{eq:bottom-pullback}, with $z=X$, for $m=1,2$.  The powers
of $X$ cancel, giving \eqref{eq:bands}; specialization at $t=t_k$ gives
\eqref{eq:dimensionbands}.
\end{proof}
\noindent
Several Lucas-Lehmer-type identities follow immediately (recall $d_k=\Psi_3(t_k) = t_k+1$); for example:
\begin{align*}
 d_{2k}-3&=(d_k-3)(d_k+1),\\
 d_{3k}-3&=(d_k-3)d_k^2,\\
 d_{5k}-3&=(d_k-3)\bigl(t_k^2+t_k-1\bigr)^2.
\end{align*}

\section[Cyclotomic norm identities: proof of the main theorem]{Cyclotomic norm identities: proof of Theorem~\ref{thm:main}}\label{sec:norms}

The $q$-algebra viewpoint described in the Introduction underpins the 
following, by keeping $X$ indeterminate and specializing
$X=\eps^k$ only after the polynomial norm identities have been proved.
We take $X$ to be fixed under the Galois groups below.
For a finite Galois extension $L/F$ and a polynomial
$G(X)\in L[X]$, write
\[
  \Nm_{L[X]/F[X]}(G)
  :=
  \prod_{\sigma\in\Gal(L/F)}\sigma(G)
  \in F[X],
\]
where $\sigma$ acts on the coefficients of $G$ and fixes $X$.
This is the usual algebra norm for the finite free extension
$F[X]\subseteq L[X]$.  When the polynomial rings are clear, we
abbreviate it to $\Nm_{L/F}(G)$.

\begin{lemma}\label{lem:CycNId}
Let $\ell$ be a prime.
\begin{enumerate}[(i)]
\item\label{lem:CycNDeep} For every $M\geq1$ with $\ell\mid M$, 
$
  \Nm_{\Q(\mu_{M\ell})/\Q(\mu_M)}
  \bigl(1-\zeta_{M\ell}X\bigr)
  =
  1-\zeta_MX^\ell
$.

\item\label{lem:CycNAdj} For every $m\geq1$ with $(m,\ell)=1$, 
$
  \Nm_{\Q(\mu_{m\ell})/\Q(\mu_m)}
  \bigl(1-\zeta_{m\ell}X\bigr)
  =
  \frac{1-\zeta_mX^\ell}{1-\zeta_m^{\ell^{-1}}X}
  \in\Q(\mu_m)[X]
$. 

\noindent
For $m=1$ this quotient is $(1-X^\ell)/(1-X)=\Phi_\ell(X)$.
\end{enumerate}
\end{lemma}

\begin{proof}
Throughout we shall use the identification
\[
 \Gal(\Q(\mu_N)/\Q)\simeq(\Z/N\Z)^\times,
 \qquad
 a\longmapsto\sigma_a,\qquad
 \sigma_a(\zeta_N)=\zeta_N^a,
\]
determined by the primitive roots fixed in
\S\ref{subsec:main-theorem}.  Under this identification, the restriction homomorphism on Galois groups
induced by the inclusion $\Q(\mu_{N_0})\subseteq\Q(\mu_N)$, for $N_0\mid N$,
corresponds to reduction of residue classes modulo~$N_0$.

For part~\textup{(\ref{lem:CycNDeep})}, the relative
Galois group is therefore
\[
 \ker\!\left((\Z/M\ell\Z)^\times\longrightarrow
             (\Z/M\Z)^\times\right)
 =\{1+Mj\pmod{M\ell}:j\pmod\ell\};
\]
all these classes are units because $\ell\mid M$.  The power-compatible
choice of roots of unity gives $\zeta_{M\ell}^{M}=\zeta_\ell$ and
$\zeta_{M\ell}^{\ell}=\zeta_M$, and hence
$\zeta_{M\ell}^{1+Mj}=\zeta_{M\ell}\zeta_\ell^j$.  Using
\(
  \prod_{j\bmod\ell}(1-\zeta_\ell^jZ)=1-Z^\ell
\)
with $Z=\zeta_{M\ell}X$ gives $1-\zeta_{M\ell}^{\ell}X^\ell
=1-\zeta_MX^\ell$.

For part~\textup{(\ref{lem:CycNAdj})}, similarly,
\[
 \ker\!\left((\Z/m\ell\Z)^\times\longrightarrow
             (\Z/m\Z)^\times\right)
 =\{1+mj\pmod{m\ell}:j\pmod\ell,\ \ell\nmid1+mj\}.
\]
Since $(m,\ell)=1$, exactly one value $j_0\pmod\ell$ is excluded.  The
power-compatible choice gives $\zeta_{m\ell}^{m}=\zeta_\ell$ and
$\zeta_{m\ell}^{\ell}=\zeta_m$, so
$\zeta_{m\ell}^{1+mj}=\zeta_{m\ell}\zeta_\ell^j$.  If
$1+mj_0=\ell b$, then $b\equiv\ell^{-1}\pmod m$.  Completing the norm
product by this missing factor gives
\[
  \prod_{j\bmod\ell}
  \bigl(1-\zeta_{m\ell}^{1+mj}X\bigr)
  =1-\zeta_mX^\ell,
\]
whereas the omitted factor is $1-\zeta_m^{\ell^{-1}}X$.  Division gives
the asserted quotient.
\end{proof}

The identities of Lemma~\ref{lem:CycNId} may be composed into a
single closed formula, which is the principal structural lemma of the
paper.  Its prime-power case reappears in
\S\ref{sec:conclusions}, where adjoining compatible Kummer division
points turns it into a symmetric pair of norm identities.

\begin{lemma}\label{lem:CompN}
Let $m,n\geq1$.  Then, in $\Q(\mu_m)[X]$,
\begin{equation}\label{eq:CompN}
 \Nm_{\Q(\mu_{mn})[X]/\Q(\mu_m)[X]}\bigl(1-\zeta_{mn}X\bigr)
 =\prod_{e\mid\rad{n'}}
  \bigl(1-\zeta_m^{e^{-1}}X^{\,n/e}\bigr)^{\mu(e)}.
\end{equation}
\end{lemma}
\begin{proof}
We need a couple of preliminary facts.  Since the cyclotomic tower is
abelian, the relative norms commute with the coefficient automorphisms;
see, for example, \cite[Chapter~2]{Washington}:
for $M_0\mid M$, $u\in(\Z/M\Z)^\times$ and 
any polynomial~$G=G(X)\in\Q(\mu_M)[X]$,
\begin{equation}\label{eq:norm-equivariance}
 \Nm_{\Q(\mu_M)/\Q(\mu_{M_0})}\bigl(\sigma_u(G)\bigr)
 =\sigma_{\bar u}\bigl(\Nm_{\Q(\mu_M)/\Q(\mu_{M_0})}(G)\bigr),
\end{equation}
where $\bar u$ is the class of $u$ modulo $M_0$; and they commute with
the substitutions $X\mapsto X^{j}$.  
Reduction modulo a divisor is also compatible with inversion: if
$\ell\lambda\equiv1\pmod M$, then the reduction of $\lambda$ modulo
$M_0$ is the inverse of $\ell$ modulo $M_0$.

We proceed by induction on $\Omega(n)$, the 
number of prime factors of $n$ counted
with multiplicity, peeling off one prime from the top of the tower and
using the transitivity of the norm. 
The base case $\Omega(n)=0$, which corresponds 
to the value $n=1$, is immediate.

So suppose first of all that $\ell \mid n$, and either $\ell\mid m$ or
$\ell^2\mid n$.  Then $\ell$ divides $mn/\ell$, and
Lemma~\ref{lem:CycNId}\textup{(\ref{lem:CycNDeep})} with $M=mn/\ell$
gives~
$
 \Nm_{\Q(\mu_{mn})/\Q(\mu_{mn/\ell})}(1-\zeta_{mn}X)
 =1-\zeta_{mn/\ell}X^{\ell}
$. 
The prime-to-$m$ radicals of $n$ and of $n/\ell$ coincide, so the
induction hypothesis for $n/\ell$, applied in the variable
$X^{\ell}$, yields \eqref{eq:CompN}.

It remains to treat the case in which $n$ is square--free and coprime to
$m$, so that $b=1$ and $\rad{n'}=n$.  Let $\ell\mid n$,
so that in particular~$\ell\nmid mn/\ell$.  By
Lemma~\ref{lem:CycNId}\textup{(\ref{lem:CycNAdj})},
\[
 \Nm_{\Q(\mu_{mn})/\Q(\mu_{mn/\ell})}(1-\zeta_{mn}X)
 =\frac{1-\zeta_{mn/\ell}X^{\ell}}{\sigma_\lambda(1-\zeta_{mn/\ell}X)},
 \qquad \ell\lambda\equiv1\pmod {mn/\ell}.
\]
Taking $\Nm_{\Q(\mu_{mn/\ell})/\Q(\mu_m)}$ of the numerator and applying the
induction hypothesis in the variable $X^{\ell}$ produces the factors of
\eqref{eq:CompN} indexed by the divisors of $n/\ell$.  For the
denominator, \eqref{eq:norm-equivariance} and the induction hypothesis
give
\begin{align*}
 \Nm_{\Q(\mu_{mn/\ell})/\Q(\mu_m)}\bigl(\sigma_\lambda(1-\zeta_{mn/\ell}X)\bigr)
 &=\sigma_{\ell^{-1}}\Bigl(\,\prod_{e\mid n/\ell}
  \bigl(1-\zeta_m^{e^{-1}}X^{(n/\ell)/e}\bigr)^{\mu(e)}\Bigr)\\
 &=\prod_{e\mid n/\ell}
  \bigl(1-\zeta_m^{(\ell e)^{-1}}X^{\,n/(\ell e)}\bigr)^{\mu(e)},
\end{align*}
since $\lambda$ reduces modulo $m$ to $\ell^{-1}$ and
$\ell^{-1}e^{-1}=(\ell e)^{-1}$ in $(\Z/m\Z)^\times$.  
As $\mu(\ell e)=-\mu(e)$, these are precisely the
factors of \eqref{eq:CompN} indexed by the divisors of $n$
divisible by $\ell$, and the two families together exhaust the divisors
of $n$.
\end{proof}

We can now prove the main theorem.

\begin{proof}[Proof of Theorem~\ref{thm:main}]
The hypothesis \(K\cap\Q(\mu_{mn})=\Q\) identifies
$
 \Gal\bigl(K(\mu_{mn})/K(\mu_m)\bigr)
 \simeq
 \Gal\bigl(\Q(\mu_{mn})/\Q(\mu_m)\bigr)$. 
In particular, every relative automorphism on the left fixes \(K\),
and hence fixes \(\eps\).  Thus Lemma~\ref{lem:CompN} remains valid
after base change to \(K\), and evaluation at \(X=\eps^k\) commutes
with taking the product of conjugates defining the norm.  Specializing
\eqref{eq:CompN} at \(X=\eps^k\) and using~$
 1-\zeta_m^{e^{-1}}\eps^{(n/e)k}
   =\sigma_{e^{-1}}\bigl(c_{m,(n/e)k}\bigr)$,
since \(\sigma_{e^{-1}}\) fixes \(K\), gives
\eqref{eq:composite-relation}.
\end{proof}

\begin{proof}[Proof of Corollary~\ref{cor:prime-step}]
For part~\textup{(\ref{cor:prime-step-deepening})}, the unique quadratic subfield of
$\Q(\mu_{\ell^{s+1}})$ is already contained in $\Q(\mu_\ell)$.  Thus
adjoining $K$ does not change the relative degree from
$\Q(\mu_{\ell^s})$ to $\Q(\mu_{\ell^{s+1}})$, and the $K$-relative
norm is obtained from
Lemma~\ref{lem:CycNId}\textup{(\ref{lem:CycNDeep})}, taken with
$M=\ell^s$, by base change; specializing $X=\eps^k$ gives~$
  \Nm_{K(\mu_{\ell^{s+1}})/K(\mu_{\ell^s})}
  \bigl(c_{\ell^{s+1},k}\bigr)
  =c_{\ell^s,\ell k}$. 
Part~\textup{(\ref{cor:prime-step-adjoining})} for $m\geq2$ is the case $n=\ell$ of
Theorem~\ref{thm:main}, the divisors $e=1,\ell$ contributing
$c_{m,\ell k}$ and $\sigma_{\ell^{-1}}(c_{m,k})^{-1}$.  For $m=1$, base
change and Lemma~\ref{lem:CycNId}\textup{(\ref{lem:CycNAdj})}
give $\Nm_{K(\mu_\ell)/K}(c_{\ell,k})=\Phi_\ell(\eps^k)$, and for
$k\geq1$ this equals $c_{1,\ell k}/c_{1,k}$.
\end{proof}

\subsection{Integer identities in the trace coordinate}
Passing to the trace coordinate turns the norm identities into the
following identities in $\Z$.  Although some right-hand sides are
written as quotients, they are integers because they equal the
left-hand side.  In particular, taking $k=1$ in~\eqref{eq:master}
gives the leading-edge inversion used in~\S\ref{sec:leading-edge}.

Recall from the introduction the decomposition \(n=bn'\).  Then
\(\rad{mb}=\rad{m}\) and \((mb,n')=1\), and hence
\begin{equation}\label{eq:totient}
 \varphi(mn)=b\varphi(m)\varphi(n')
 =\varphi(m)\sum_{e\mid\rad{n'}}\mu(e)\frac ne,
\end{equation}
since the sum is
\(n\prod_{\ell\mid n'}(1-1/\ell)=b\varphi(n')\), where~$\ell$ runs over primes only.
We shall use \eqref{eq:totient} below when passing to the trace
coordinate; recall also that inverses are taken modulo~\(m\).

\begin{proposition}\label{prop:master}
For $m\geq3$, $n\geq1$ and $k\geq0$,
\begin{equation}\label{eq:master}
 \Psi_{mn}(t_k)
 =\prod_{e\mid\rad{n'}}
  \Psi_m\bigl(t_{(n/e)k}\bigr)^{\mu(e)},
\end{equation}
where~$e$ runs over \emph{all} divisors of~$\rad{n'}$. 
In particular:
\begin{enumerate}
\setlength{\itemsep}{0pt}
\item[(i)] For a prime $\ell$ with $\ell^s\geq3$,
\(\Psi_{\ell^{s+1}}(t_k)=\Psi_{\ell^s}(t_{\ell k})\).
\item[(ii)] For a prime $\ell\nmid m$, with $m\geq3$,
\(\displaystyle\Psi_{m\ell}(t_k)=
\Psi_m(t_{\ell k})/\Psi_m(t_k)\).
\item[(iii)] For square--free $r$ with $(r,m)=1$,
\(\displaystyle\Psi_{mr}(t_k)=
\prod_{a\mid r}\Psi_m(t_{ak})^{\mu(r/a)}\).
\end{enumerate}
\end{proposition}

\begin{proof}
The standard prime-step identities for cyclotomic 
polynomials are~\cite[Chapter~2]{Washington} $
 \Phi_M(X^{\ell})=\Phi_{M\ell}(X)$ if~$\ell\mid M$, 
 and~$ \Phi_M(X^{\ell})=\Phi_M(X)\Phi_{M\ell}(X)$ if~$\ell\nmid M$. 
Iterating these identities exactly as in the proof of
Lemma~\ref{lem:CompN} gives
\[
 \Phi_{mn}(X)
 =\prod_{e\mid\rad{n'}}
  \Phi_m\bigl(X^{\,n/e}\bigr)^{\mu(e)}.
\]
Evaluating at $X=\eps^k$ and inserting the normalizing powers of
$\eps$ through \eqref{eq:pal}, the exponents balance by~\eqref{eq:totient}.  The
displayed special cases are $n=\ell$ with $\ell\mid m$, $n=\ell$ with
$\ell\nmid m$, and $n=r$; integrality follows from the left-hand side of~\eqref{eq:master}.
\end{proof}

\subsection[The basic example D=2]{The basic example \texorpdfstring{$D=2$}{D=2}}
We now state and apply a non-trivial identity between 
levels~3 and~5, which is really just two ways of
factoring $\Phi_{15}$ from a prime shift:
\begin{equation}\label{eq:15two}
 \Psi_{15}(t_k)
 =\frac{\Psi_5(t_{3k})}{\Psi_5(t_k)}
 =\frac{\Psi_3(t_{5k})}{\Psi_3(t_k)}.
\end{equation}

Let $K=\Q(\sqrt2)$, $u=1+\sqrt2$ and $\eps=u^2=3+2\sqrt2$, so that
$t_1=6$ and $d_1=7$.  Here~$
 \Psi_{15}(X)=X^4-X^3-4X^2+4X+1$ and so~$\Psi_{15}(t_1)=961=31^2$. 
Since $o(31)=15$ and $a_{31}=2$, one has
$\rho_3(31)=5$, $\rho_5(31)=3$ and $\rho_{15}(31)=1$, with
\begin{align*}
 \Psi_3(t_5)&=6727=7\cdot31^2,\\
 \Psi_5(t_3)&=39401=41\cdot31^2,\\
 \Psi_{15}(t_1)&=31^2.
\end{align*}
Thus the single congruence $31^2\mid\eps^{15}-1$ appears at levels
$3$, $5$ and $15$, compatibly with~\eqref{eq:15two}.

\section{Exact-order support and level-dependent ranks of apparition}\label{sec:support}

\begin{proposition}\label{prop:support}
Let $m\geq1$ and assume that $p\nmid2mD$.  Then $\rho_m(p)<\infty$ if and
only if $m\mid o(p)$.  In that case~$\rho_m(p)=\frac{o(p)}{m}$, 
and, for every $k\geq1$, $p\mid\Psi_m(t_k)$ if and 
only if~$k=\rho_m(p)l$ for some~$l\geq1$ with~$(l,m)=1$. 
\end{proposition}

\noindent See equation~(121) of~\cite{BGM} for a direct proof in the SIC case~$m=3$. 

\begin{proof}
For $m\geq3$, reducing~\eqref{eq:pal}---evaluated 
at $X=\eps^k$---modulo $p$, and using
the standard description of the roots of cyclotomic polynomials in
characteristic $p$ gives
$p\mid\Psi_m(t_k)$ if and only if~$\bar\eps_p^{\,k}$ has exact order
$m$; see \cite[Chapter~2, \S4]{LidlNiederreiter}.
For $m=1,2$, the same conclusion follows from
\eqref{eq:bottom}.  Since $\bar\eps_p$ has order $o(p)$,
$\ord(\bar\eps_p^{\,k})=o(p)/\gcd(o(p),k)$.  Thus $p$ occurs in the sequence 
precisely when $m\mid o(p)$.  Writing $o(p)=m\rho$, the
condition that $\bar\eps_p^{\,k}$ have order $m$ is
$\gcd(m\rho,k)=\rho$; or equivalently~$k=\rho l$ 
with $(l,m)=1$.  The least such positive index 
is $\rho=o(p)/m$, proving the result.
\end{proof}

This divisor description of prime occurrences 
also conserves multiplicities: the ordinary
case is $a_p=1$, while a Wieferich or Wall--Sun--Sun excess is the
condition $a_p\geq2$. 
Here we give the valuation for any occurrence of~$p$ 
at a level~$m\geq3$. 

\begin{proposition}\label{prop:valuation}
Let $m\geq3$, let $p\nmid2mD$, and suppose that
$p\mid\Psi_m(t_k)$.  Then
\begin{equation}\label{eq:valuation}
 v_p\bigl(\Psi_m(t_k)\bigr)
 =a_p+v_p\!\left(\frac{k}{\rho_m(p)}\right).
\end{equation}
\end{proposition}
\noindent 
Since $o(p)\mid p-\leg{D}{p}$, one has $p\nmid\rho_m(p)$, so
\eqref{eq:valuation} may equally be written
$v_p\bigl(\Psi_m(t_k)\bigr)=a_p+v_p(k)$. 
In particular, at the first occurrence,~
$ v_p\bigl(\Psi_m(t_{\rho_m(p)})\bigr)=a_p$. 
See Proposition A.3 of~\cite{BGM} for the SIC ($m=3$) case.

\begin{proof}
Fix a prime \(\mathfrak p\mid p\mathcal O_K\).
Since \(\bar\eps_p^{\,k}\) has exact order \(m\), among the factors in~$
 (\eps^k)^m-1=\prod_{e\mid m}\Phi_e(\eps^k)$, 
only \(\Phi_m(\eps^k)\) is divisible by \(\mathfrak p\).  
Hence~$v_{\mathfrak p}(\Phi_m(\eps^k)) = v_{\mathfrak p}(\eps^{km}-1)$. 
By Proposition~\ref{prop:support}, \(km=o(p)l\), where
\(l=k/\rho_m(p)\).  Write \(l=p^r a\) with \(p\nmid a\).  Since
\(\eps^{o(p)}\equiv1\pmod{\mathfrak p}\),
\[
 \frac{\eps^{o(p)a}-1}{\eps^{o(p)}-1}
 =1+\eps^{o(p)}+\cdots+\eps^{o(p)(a-1)}
 \equiv a\pmod{\mathfrak p},
\]
so~$v_{\mathfrak p}(\eps^{o(p)a}-1)= v_{\mathfrak p}(\eps^{o(p)}-1)=a_p$. 
Now, if \(y\equiv1\pmod{\mathfrak p}\), then, since
\(K_{\mathfrak p}/\Q_p\) is unramified and \(p>2\), the binomial
theorem gives~$v_{\mathfrak p}(y^p-1)=v_{\mathfrak p}(y-1)+1$. 
Applying this successively \(r\) times, starting with 
\(y=\eps^{o(p)a}\), gives
\[
 v_{\mathfrak p}(\eps^{o(p)l}-1)
 =a_p+r
 =a_p+v_p(l)
 =a_p+v_p\!\left(\frac{k}{\rho_m(p)}\right).
\]
Since \(p\nmid D\), the normalized valuation \(v_{\mathfrak p}\)
restricts to \(v_p\) on \(\Q\); now \eqref{eq:pal} shows that
\(\Phi_m(\eps^k)\) differs from \(\Psi_m(t_k)\) by a global unit.
\end{proof}

\begin{corollary}\label{cor:levelsupport}
Let $p\nmid2D$.  For every
divisor $m\geq3$ of $o(p)$,
$p^{a_p}\parallel\Psi_m\bigl(t_{\rho_m(p)}\bigr)$. 
\end{corollary}

Thus every level $m\geq3$ dividing $o(p)$ displays the full valuation
$a_p$ at its first occurrence; in particular,
$p^{a_p}\parallel\Psi_{o(p)}(t_1)$ when $o(p)\geq3$.  At the two
branch levels the multiplicity doubles:
\begin{proposition}\label{prop:bottom-levels}
Let $p\nmid2D$.  Then~$\rho_1(p)=o(p)$, while
$\rho_2(p)=o(p)/2$ if $2\mid o(p)$ and $\rho_2(p)=\infty$ otherwise.
Whenever $m\in\{1,2\}$ and $\rho_m(p)<\infty$, one has, at every
occurrence $k$ of $p$ in the level-$m$ sequence,
\begin{equation}\label{eq:bottom-valuation}
 v_p\bigl(\Psi_m(t_k)\bigr)
 =2\left(
     a_p+v_p\!\left(\frac{k}{\rho_m(p)}\right)
   \right).
\end{equation}
\end{proposition}
\begin{proof}
The assertions concerning $\rho_1(p)$ and $\rho_2(p)$ are the cases
$m=1,2$ of Proposition~\ref{prop:support}.  For $m=1$, write
$k=o(p)l=\rho_1(p)l$.  From \eqref{eq:bottom},
$\Psi_1(t_k)=\eps^{-k}(\eps^k-1)^2$, and local lifting of the exponent
gives $v_{\mathfrak p}(\eps^k-1)=a_p+v_p(l)$.  
For $m=2$, write $k=(o(p)/2)l=\rho_2(p)l$; by
Proposition~\ref{prop:support}, $l$ is odd.  Since
$\bar\eps_p^{\,k}=-1$, the factor $\eps^k-1$ is a local unit, and
$v_{\mathfrak p}(\eps^k+1)=v_{\mathfrak p}(\eps^{2k}-1)=a_p+v_p(l)$.
Squaring in the two identities \eqref{eq:bottom} proves
\eqref{eq:bottom-valuation}.
\end{proof}

The factor~$2$ in~\eqref{eq:bottom-valuation} is the branch
multiplicity of $z\mapsto z+z^{-1}$ at $z=\pm1$; compare
\eqref{eq:bottom-branch}.
In particular, $v_p(t_{o(p)}-2)=2a_p$, and 
when $2\mid o(p)$, $v_p(t_{o(p)/2}+2)=2a_p$; 
we use this in~Appendix~\ref{app:wss-tables}.

\section{Leading-edge factors and primitive prime support}\label{sec:leading-edge}
Fix $m\geq3$ and 
denote~$\mathcal N_m(k):=\Psi_{mk}(t_1)$ for~$k\geq1$. 
Now we decompose~$k$ as~$k=br$, 
where $b$ is the maximal divisor of $k$ supported on the
primes dividing $m$ and $(r,m)=1$.  M\"obius inversion in the exponent
index then isolates first occurrences at level~$m$.

\begin{proposition}
\label{prop:leadingmobius}
With the notation above,
\begin{equation}\label{eq:branchfactorization}
 \Psi_m(t_{br})=\prod_{e\mid r}\mathcal N_m(be).
\end{equation}
Equivalently,
\begin{equation}\label{eq:leadingmobius}
 \mathcal N_m(br)=\prod_{e\mid r}\Psi_m(t_{be})^{\mu(r/e)}.
\end{equation}
\end{proposition}

\begin{proof}
Equation \eqref{eq:leadingmobius} is the master identity
\eqref{eq:master}, applied at $t_1$ with tower ratio $br$ over the base
level $m$, after reindexing the divisors.
Ordinary multiplicative
M\"obius inversion on the divisor lattice of $r$ then gives
\eqref{eq:branchfactorization}.
\end{proof}

For the Wall--Sun--Sun question we need only the prime-to-$2Dmk$ part
of the leading edge.  Put
$\mathcal N_m^{\circ}(k):=\prod_{p\nmid2Dmk}p^{v_p(\mathcal N_m(k))}$.
By Propositions~\ref{prop:support} and~\ref{prop:valuation},
\begin{equation}\label{eq:cleanleadingedge-product}
 \mathcal N_m^{\circ}(k)
 =\prod_{\substack{p\nmid2D\\o(p)=mk}}p^{a_p}
 =\prod_{\rho_m(p)=k}p^{a_p}.
\end{equation}
Consequently, if
$E_m(k):=\mathcal N_m^{\circ}(k)/\rad{\mathcal N_m^{\circ}(k)}$, then
\[
 E_m(k)=\prod_{\rho_m(p)=k}p^{a_p-1},
 \qquad
 p\mid E_m(k)\Longleftrightarrow \rho_m(p)=k\ \text{and }a_p\geq2.
\]
Thus $E_m(k)$ isolates the primes with $a_p\geq2$ at first
occurrence, with exponent $a_p-1$.

\begin{theorem}\label{thm:primitive-orders}
For every $n>12$ there is a prime $p\nmid2D$ with $o(p)=n$. 

Equivalently, for every $m\geq3$ and $k\geq1$ with $mk>12$,
one has $\mathcal N_m^{\circ}(k)>1$.
\end{theorem}

See Proposition A.2 of~\cite{BGM} for a more direct proof of this 
result in the SIC ($m=3$) case.

\begin{proof}
Put $w_n:=\prod_{\substack{e\mid n\\ e\geq3}}\Psi_e(t_1)\in\Z_{>0}$,
so that, by the factorization~\eqref{eq:bands} evaluated at $t=t_1$,
\[
 w_n^2=\frac{t_n-2}{(t_1-2)(t_1+2)^{\mathbf{1}_{2\mid n}}}.
\]
The pair $(\alpha,\beta)=(\eps^{1/2},\eps^{-1/2})$ is a real Lehmer
pair: $(\alpha+\beta)^2=t_1+2$ and $\alpha\beta=1$ are coprime
integers, and $\alpha/\beta=\eps$ is not a root of unity.  Its Lehmer
numbers are
\[
 \widetilde u_n=
 \begin{cases}
  \displaystyle\frac{\alpha^n-\beta^n}{\alpha-\beta},
       & n\ \text{odd},\\[6pt]
  \displaystyle\frac{\alpha^n-\beta^n}{\alpha^2-\beta^2},
       & n\ \text{even}.
 \end{cases}
\]
Consequently,
\[
 \widetilde u_n^2=
 \begin{cases}
  \displaystyle\frac{t_n-2}{t_1-2},
       & n\ \text{odd},\\[6pt]
  \displaystyle\frac{t_n-2}{(t_1-2)(t_1+2)},
       & n\ \text{even}.
 \end{cases}
\]
Since $\alpha>\beta>0$, comparison with the preceding factorization
gives $w_n=\widetilde u_n$.

By the primitive-divisor theorem for real Lehmer sequences~\cite{Carmichael1913, Durst1959, Ward1955},
\(\widetilde u_n\) has a primitive divisor for every \(n>12\):
a prime \(p\) such that
\[
 p\mid w_n,\qquad
 p\nmid(\alpha^2-\beta^2)^2=t_1^2-4
       =\Delta[\OK:\Z[\eps]]^2,
 \qquad
 p\nmid w_j\quad(0<j<n).
\]
Hence \(p\nmid D\).  Moreover \(p\neq2\): if \(t_1\) is even, then
\(2\mid t_1^2-4\), while if \(t_1\) is odd, then
\(2\mid w_3=\Psi_3(t_1)=t_1+1\), and \(3<n\).  Thus \(p\nmid2D\).

From $p\mid w_n$ and the
displayed square identity, $p\mid t_n-2$, so $o(p)\mid n$ by
Proposition~\ref{prop:bottom-levels}; and if $o(p)<n$, then
$p\mid t_{o(p)}-2
=w_{o(p)}^2\,(t_1-2)(t_1+2)^{\mathbf{1}_{2\mid o(p)}}$
together with $p\nmid t_1^2-4$ would force $p\mid w_{o(p)}$,
contradicting primitivity.  Hence $o(p)=n$.

For the second statement, a prime with $o(p)=mk$ automatically
satisfies $p\nmid mk$, since $mk\mid p-\leg{D}{p}$; it therefore
appears in $\mathcal N_m^{\circ}(k)$ with exponent $a_p\geq1$ by
\eqref{eq:cleanleadingedge-product}.
\end{proof}

\begin{proof}[Proof of Theorem~\ref{thm:A}]
Part~\textup{(\ref{thmA:parametrization})} was proved in
\S\ref{subsec:cross-level}, and part~\textup{(\ref{thmA:chebyshev})}
is~\eqref{eq:dimensionbands}, with the branch factors identified in
\S\ref{subsec:trace-quotient}.  Part~\textup{(\ref{thmA:support})}
is Proposition~\ref{prop:support}, Proposition~\ref{prop:valuation}
and Corollary~\ref{cor:levelsupport}; part~\textup{(\ref{thmA:orders})}
is Theorem~\ref{thm:primitive-orders}.
\end{proof}

\section{Concluding remarks}\label{sec:conclusions}
Two M\"obius structures have appeared.  Proposition~\ref{prop:leadingmobius}
gives ordinary divisor-lattice inversion in the $k$-variable. 
Indeed, after
taking logarithms, \eqref{eq:branchfactorization} becomes a
Dirichlet-convolution identity whose leading archimedean growth is
governed by the regulator, since for each fixed $m\geq3$,
\[
 \eps^{-k\varphi(m)/2}\Psi_m(t_k)\longrightarrow1
 \qquad(k\longrightarrow\infty).
\]
The composite norm relation
carries the corresponding M\"obius structure in the cyclotomic
variable, twisted by the coefficient automorphisms
$\sigma_{e^{-1}}$.  The two directions are coupled by the power
substitution $k\mapsto pk$ in
Corollary~\ref{cor:prime-step}\textup{(i)}.

If \(p\nmid2D\), fix \(\mathfrak p\mid p\mathcal O_K\) and put
\(\eta_p:=\eps^{o(p)}\).
By definition~$o(p)$ is a $p$-adic
unit and~$\eta_p\in1+\pp$.  Since $p$ is odd and
$K_{\mathfrak p}/\Q_p$ is unramified, the $p$-adic logarithm identifies
$1+\pp$ with $\pp$ and gives
\[
 v_{\pp}(\log_p\eps)=v_{\pp}(\log_p\eta_p)
 =v_{\pp}(\eta_p-1)=a_p.
\]
The same logarithm gives the local Kummer interpretation
\begin{equation}\label{eq:kummer-depth}
 a_p-1
 =\max\bigl\{j\geq0:\eta_p\in K_{\mathfrak p}^{\times p^j}\bigr\}.
\end{equation}
Thus $p$ is Wall--Sun--Sun precisely when
$\eta_p\in K_{\mathfrak p}^{\times p}$; more generally, the exponent
$a_p-1$ occurring in $E_m(k)$ at a first occurrence is exactly the
local $p$-divisibility depth of $\eps^{o(p)}$.  Write $h_K$ for the
ordinary class number of~$K$, and write $L_p(s)$ for the
Kubota--Leopoldt $p$-adic $L$-function attached to the primitive
quadratic Dirichlet character of~$K$, with the normalization of
\cite[Chapter~5]{Washington}.  The $p$-adic class number
formula~\cite{coates} and \cite[Thm.~5.24]{Washington} then gives
\[
 v_p\bigl(L_p(1)\bigr)=v_p(h_K)+a_p-1
 \qquad(p\nmid2\Delta).
\]
Hence the same integer $a_p-1$ measures both this local Kummer depth
and, up to the class number, the first nontrivial $p$-adic divisibility
of the corresponding $L$-value.

The norm relations also suggest that we should 
introduce fractional exponents of~$\eps$.
At the fixed real embedding one has $\eps>1$, so the positive root
$\eps^x$ is canonically defined for every $x\in\Q$.  Thus the exponents
are indexed by $\Q$; at a fixed prime $p$ the $p$-primary exponents lie
in $\Z[1/p]$.  

To place the positive real division points in a local algebraic
closure, fix an embedding
$\overline{\Q}\hookrightarrow\overline{\Q}_p$ inducing $\pp$, and use
the same notation for their images.  Put
\[
 \eta_{p,s}:=\eta_p^{1/p^s}=\eps^{o(p)/p^s},\qquad
 \eta_{p,0}=\eta_p,\qquad
 \eta_{p,s+1}^p=\eta_{p,s}\quad(s\geq0).
\]
Replacing this compatible chain by another one multiplies its terms by
compatible $p$-power roots of unity, and hence does not change the
compositum after adjoining $\mu_{p^\infty}$.  Put
\[
 \mathcal F_{p,\infty}
 :=K_{\mathfrak p}\bigl(\mu_{p^\infty},\eta_p^{1/p^\infty}\bigr).
\]
This is a local analogue of the standard false Tate curve extension;
compare \cite[(1.1)]{KimFalseTate} and
\cite[\S3.3]{HV03}.
Write
\[
 G_p:=\Gal(\mathcal F_{p,\infty}/K_{\mathfrak p}),\qquad
 H_p:=\Gal(\mathcal F_{p,\infty}/K_{\mathfrak p}(\mu_{p^\infty})),
\]
and
\[
 \Gamma_p:=\Gal(K_{\mathfrak p}(\mu_{p^\infty})/K_{\mathfrak p}).
\]
Since $K_{\mathfrak p}/\Q_p$ is unramified whereas the cyclotomic
tower is totally ramified, $\Gamma_p\simeq\Z_p^\times$.  
Passing to the limit in the Kummer direction, $H_p\simeq\Z_p$, and the standard false
Tate description is
\[
 G_p\simeq H_p\rtimes_{\chi_p}\Gamma_p
 \simeq \Z_p\rtimes_{\chi_p}\Z_p^\times.
\]
Here $H_p$ is the normal Kummer subgroup and $\Gamma_p$ acts through
the cyclotomic character: if $\kappa\in H_p$ and
$\gamma\in\Gamma_p$, 
then~$\gamma\kappa\gamma^{-1}=\kappa^{\chi_p(\gamma)}$.

We introduce some abbreviated notation: write $\zeta_r:=\zeta_{p^r}$ and, for $r\geq1$ and $s\geq0$, put
\[
 F_{r,s}:=K_{\mathfrak p}(\zeta_r,\eta_{p,s}) \textrm{\ \ and\ \ }
 c^{(p)}_{r,s}:=1-\zeta_r\eta_{p,s}.
\]
The Chebyshev shift of Corollary~\ref{cor:prime-step}\textup{(i)} disappears on
this divided grid in a particularly symmetric way, as we now show. 
Bear in mind that the subscripts~$r$ and~$s$ in the following really represent 
respectively~$p^r$ and~$p^s$ (with the obvious extensions). 

\begin{proposition}\label{prop:false-tate-norms}
For $r\geq1$ and $s\geq0$ one has
\begin{equation}\label{eq:false-tate-resultants}
 \Res_Y(Y^p-\zeta_r,1-Y\eta_{p,s+1})
 =\Res_Y(Y^p-\eta_{p,s},1-\zeta_{r+1}Y)
 =c^{(p)}_{r,s}.
\end{equation}
Whenever the two adjacent extensions have their full degree $p$, this
is the pair of field-norm identities
\begin{equation}\label{eq:false-tate-norms}
 \Nm_{F_{r+1,s+1}/F_{r,s+1}}(c^{(p)}_{r+1,s+1})
 =\Nm_{F_{r+1,s+1}/F_{r+1,s}}(c^{(p)}_{r+1,s+1})
 =c^{(p)}_{r,s}.
\end{equation}
\end{proposition}

\begin{proof}
The roots of $Y^p-\zeta_r$ are $\xi\zeta_{r+1}$ and those of
$Y^p-\eta_{p,s}$ are $\xi\eta_{p,s+1}$, with $\xi^p=1$.  Thus either
resultant in \eqref{eq:false-tate-resultants} 
is~$\prod_{\xi^p=1}(1-\xi\zeta_{r+1}\eta_{p,s+1}) =1-\zeta_r\eta_{p,s}$. 
If the corresponding adjunction has degree $p$, these products are
exactly the two relative field norms.
\end{proof}

Equation~\eqref{eq:kummer-depth} is a statement about
local division: it says that $a_p-1$ successive $p$-divisions of
$\eta_p$ can be made inside $K_{\mathfrak p}$.  These divisions need
not be the images of the positive real roots $\eta_{p,s}$ chosen
above.  Equivalently, choose $\beta_p\in K_{\mathfrak p}^\times$ with
$\beta_p^{p^{a_p-1}}=\eta_p$ and
$\beta_p\notin K_{\mathfrak p}^{\times p}$.  Then
\[
 K_{\mathfrak p}\bigl(\mu_{p^\infty},\eta_p^{1/p^\infty}\bigr)
 =K_{\mathfrak p}\bigl(\mu_{p^\infty},\beta_p^{1/p^\infty}\bigr),
\]
so $a_p-1$ is exactly the number of initial local Kummer divisions of
$\eta_p$ which already exist in $K_{\mathfrak p}$.  
In particular, as is well known, the 
Wall--Sun--Sun condition is exactly the
vanishing of the first local Kummer class
$[\eta_p]\in K_{\mathfrak p}^{\times}/K_{\mathfrak p}^{\times p}$.

Thus the integral substitution $k\mapsto pk$ is replaced, after
adjoining compatible division points, by two equal partial norms in a
genuine Kummer--cyclotomic square.  The resultant identity
\eqref{eq:false-tate-resultants} remains valid without a degree
hypothesis; the field-norm interpretation \eqref{eq:false-tate-norms}
applies precisely at the non-collapsed squares.  Together with
\S\ref{sec:leading-edge}, this places exact-order extraction and local
Kummer depth in the same two-variable framework. 
The fields \(F_{r,s}\) are finite layers of the local false Tate
extension \(\mathcal F_{p,\infty}/K_{\mathfrak p}\).  
This points at a future interpolation of the elements \(c^{(p)}_{r,s}\).

\section*{Acknowledgements}
I would like to thank Myungshik Kim, Terry Rudolph and the 
QOLS Group at Imperial College for their ongoing hospitality. 

In addition, I acknowledge assistance from the generative-AI systems
Anthropic's Claude Fable~5 and OpenAI's ChatGPT: the former principally
with exploratory computations, statistical analyses and an early
$q$-cyclotomic draft, and the latter principally with checking the
mathematical arguments and with reorganising and editing the material
into the present manuscript.  I take full responsibility for all
mathematical claims, computations and the final text.

\appendix
\setcounter{section}{0}
\section[Wall--Sun--Sun primes for small real quadratic fields]{\texorpdfstring{Wall--Sun--Sun primes~$p<1.7415\times10^{13}\text{ for }\Q(\qD),\ 1<D<50,\ p\nmid2D$}
{Wall--Sun--Sun primes for small real quadratic fields}}
\label{app:wss-tables}

The following are the results of an exhaustive search for WSS primes
$p<1.7415\times10^{13}$ in the first~30 real quadratic fields.  Since
$o(p)\mid p-\leg{D}{p}$ and
$p\nmid\bigl(p-\leg{D}{p}\bigr)/o(p)$,
Proposition~\ref{prop:bottom-levels} gives~$
 v_p\!\left(t_{p-\leg{D}{p}}-2\right)=2a_p$. 
Thus the exact Wall--Sun--Sun condition is~$
t_{p-\leg{D}{p}}\equiv2\pmod{p^4}$. 
But because the valuation is even, reduction modulo $p^3$ already detects
all candidates.

\begin{center}
\begingroup
\scriptsize
\renewcommand{\arraystretch}{1.12}
\setlength{\arraycolsep}{2.0pt}
\(
\begin{array}{r|r|r|l@{\hspace{6pt}}r|r|r|l}
 D & \Delta & t_1 & p & D & \Delta & t_1 & p\\
\hline
 2 & 8 & 6 & 13,\ 31,\ 1546463
   & 26 & 104 & 102 & 2683,\ 3967,\ 18587\\
 3 & 12 & 4 & 103,\ 2297860813
   & 29 & 29 & 27 & 3^{\dagger},\ 11\\
 5 & 5 & 3 & \text{---}
   & 30 & 120 & 22 & \text{---}\\
 6 & 24 & 10 & 7,\ 523,\ 4398403538003
   & 31 & 124 & 3040 & 157,\ 261687119\\
 7 & 28 & 16 & 1347680707,\ 31678908887
   & 33 & 33 & 46 & 29,\ 37,\ 6713797\\
 10 & 40 & 38 & 191,\ 643,\ 134339,\ 25233137
    & 34 & 136 & 70 & 37,\ 547,\ 4733\\
 11 & 44 & 20 & \text{---}
    & 35 & 140 & 12 & 23,\ 577,\ 1325663\\
 13 & 13 & 11 & 241
    & 37 & 37 & 146 & 7,\ 89,\ 257,\ 631,\ 18595587053\\
 14 & 56 & 30 & 6707879,\ 93140353,\ 1498255181
    & 38 & 152 & 74 & 5\\
 15 & 60 & 8 & 181,\ 1039,\ 2917,\ 2401457,\ 10521218089
    & 39 & 156 & 50 & 5,\ 7,\ 37,\ 163409,\ 795490667\\
 17 & 17 & 66 & \text{---}
    & 41 & 41 & 4098 & 29^{\dagger},\ 53,\ 7211,\ 8456649193033\\
 19 & 76 & 340 & 79,\ 1271731,\ 13599893,\ 31352389,\ 10895901685667
    & 42 & 168 & 26 & 5,\ 43,\ 71,\ 22907114557\\
 21 & 21 & 5 & 46179311
    & 43 & 172 & 6964 & 3,\ 479\\
 22 & 88 & 394 & 43,\ 73,\ 409,\ 28477
    & 46 & 184 & 48670 & 70339124081\\
 23 & 92 & 48 & 7,\ 733
    & 47 & 188 & 96 & 5762437\\
\end{array}
\)
\endgroup
\end{center}
All of the examples have $a_p=2$, other than those
marked~$^\dagger$, where $a_p=3$.

\section{Computational tests for the Wall--Sun--Sun condition}
\label{app:null-tests}
The constructions of this paper detect the condition \(a_p\geq2\),
and determine its exact valuation, but do not predict when it occurs.
We therefore tested both the event \(a_p\geq2\) and the excess
\(a_p-1\) for dependence on several simpler quantities: the regulator
\(\log\eps\) or~$\frac{1}{2}\log\eps$; the splitting type of \(p\), including comparisons at
fixed order \(o(p)\); the archimedean height of the unit; height
quotients
\[
 \frac{\eps_{DE}}{\eps_D\eps_E}
\]
in biquadratic fields \(\Q(\qD,\sqrt E)\), where \(\eps_F\)
denotes the chosen totally positive unit of \(\Q(\sqrt F)\); and
numerical functionals arising from hyperbolic-geometric dilogarithm
identities.

The statistical corpus used the subrange \(p<10^{13}\) of the
exhaustive searches reported in Appendix~\ref{app:wss-tables}; a
separate, more detailed set of tests for \(K=\Q(\sqrt2)\) and
\(p<10^{10}\);
approximately \(4\times10^8\) certified discriminant--prime records,
including every square--free \(D<10^8\) at four fixed primes~$5,7,13,31$ 
and every~\(D<10^7\) at thirty-six further primes \(p\leq3001\);
and \(902{,}688\) biquadratic triples whose unit heights span two
orders of magnitude.

Across these data sets we found no reproducible evidence that the
occurrence or size of the Wall--Sun--Sun excess depends on any of the
quantities tested.  The natural benchmark is an incidence of order
\(1/p\), corresponding heuristically to one further independent
congruence modulo \(p\).  The observed frequencies were compatible
with this benchmark, while models incorporating the regulator, the
unit height or the splitting data produced no stable improvement.
The particular hyperbolic-geometric quantities tested likewise gave
no detectable separation between the cases \(a_p=1\) and
\(a_p\geq2\).

These conclusions are empirical and are restricted to the stated
statistics and computational ranges; in particular, they do not
exclude the existence of subtler arithmetic predictors.  By the
\(p\)-adic class-number formula in \S\ref{sec:conclusions}, when
\(p\nmid h_K\) one has
\[
 a_p\geq2
 \quad\Longleftrightarrow\quad
 v_p\bigl(L_p(1)\bigr)\geq1.
\]
Thus the same computations are compatible with the heuristic that
the first relevant \(p\)-adic digit of \(L_p(1)\) behaves randomly
across the tested fields.  No distribution theorem is asserted here.
The underlying data, programs and analysis logs are archived with
the author.

\section{The case
\texorpdfstring{$K\subseteq\Q(\mu_m)$}
{K contained in a cyclotomic field}}\label{app:intersection}

Theorem~\ref{thm:main} was stated under the generic hypothesis
$K\cap\Q(\mu_{mn})=\Q$, so that the coefficient automorphisms
$\sigma_{e^{-1}}$ fix $K$.  If instead $K\subseteq\Q(\mu_m)$, then
\[
 K(\mu_m)=\Q(\mu_m),\qquad K(\mu_{mn})=\Q(\mu_{mn}),
\]
and Lemma~\ref{lem:CompN} may be evaluated directly at
$X=\eps^k$.  Thus
\[
 \Nm_{K(\mu_{mn})/K(\mu_m)}(c_{mn,k})
 =
 \prod_{e\mid\rad{n'}}
 \bigl(1-\zeta_m^{e^{-1}}\eps^{(n/e)k}\bigr)^{\mu(e)}.
\]
This is the same formula as in Theorem~\ref{thm:main}, except that
its factors should not in general be written as
$\sigma_{e^{-1}}(c_{m,(n/e)k})$.

In particular, if $m\geq2$, $(m,\ell)=1$, and
$K\subseteq\Q(\mu_m)$, then
\[
 \Nm_{K(\mu_{m\ell})/K(\mu_m)}(c_{m\ell,k})
 =
 \frac{1-\zeta_m\eps^{\ell k}}
      {1-\zeta_m^{\ell^{-1}}\eps^k}.
\]
No modification is needed in
Corollary~\ref{cor:prime-step}\textup{(i)}.  The integer identities
of Proposition~\ref{prop:master} and all subsequent
results are unchanged.  The remaining possibility, in which $K$
first enters the cyclotomic tower between the two levels under
consideration, is not needed here.

\end{document}